\documentclass[11pt]{amsart}
\usepackage{amssymb}
\theoremstyle{plain}
\newtheorem{theorem}{Theorem}
\newtheorem{corollary}{Corollary}

\theoremstyle{example}
\newtheorem{example}{Example}
\theoremstyle{definition}

\theoremstyle{remark}

\numberwithin{equation}{section}

\setbox0=\hbox{$+$}
\newdimen\plusheight
\plusheight=\ht0
\def\+{\;\lower\plusheight\hbox{$+$}\;}

\setbox0=\hbox{$-$}
\newdimen\minusheight
\minusheight=\ht0
\def\-{\;\lower\minusheight\hbox{$-$}\;}

\setbox0=\hbox{$\cdots$}
\newdimen\cdotsheight
\cdotsheight=\plusheight
\def\cds{\lower\cdotsheight\hbox{$\cdots$}}

\begin{document}
\title[  continued fractions of the form
 $K_{n=1}^{\infty} a_{n}/1$ .]
 {A convergence theorem for continued fractions of the form
 $K_{n=1}^{\infty} a_{n}/1$. }

\author{J. Mc Laughlin}
\address{Department of Mathematics\\
 Trinity College\\
300 Summit Street, Hartford, CT 06106-3100}
\email{james.mclaughlin@trincoll.edu}
\author{ Nancy J. Wyshinski}
\address{Department of Mathematics\\
       Trinity College\\
        300 Summit Street, Hartford, CT 06106-3100}
\email{nancy.wyshinski@trincoll.edu}
\date{April 14, 2003}

\begin{abstract}

In this paper we present a convergence theorem for continued
fractions of the form $K_{n=1}^{\infty}a_{n}/1$. By deriving
conditions on the $a_{n}$ which ensure that the odd and even parts
of $K_{n=1}^{\infty}a_{n}/1$ converge, these same conditions also
ensure that they converge to the same limit. Examples will be
given.
\end{abstract}

\maketitle

{\it  This paper is dedicated to Professor Olav Njastad on the
occasion of his 70th birthday.}

\section{Introduction}

In this paper we derive a convergence theorem for continued
fractions of the form $K_{n=1}^{\infty}a_{n}/1$. This is achieved
by using  Worpitzky's theorem to derive conditions on the $a_{n}$
which ensure that the odd and even parts of
$K_{n=1}^{\infty}a_{n}/1$ converge and, furthermore,
  converge to the same limit. We will assume
$a_{n} \not =0$ for any $n$, since otherwise the continued
fraction is finite and converges trivially in $\hat{\mathbb{C}}$.

\medskip

We begin by summarizing definitions and basic properties for
continued fractions that are needed.
We write $\displaystyle{A_{N}/B_{N}}$ (the $N$-th
\emph{approximant})
 for the  finite continued
fraction $b_0+K_{n = 1}^{N}a_{n}/b_{n}$ written as a rational function of the variables
$a_{1},\dots,a_{N}$,$b_1$, $\dots$, $b_N$. It is elementary that the $A_{N}$
(the $N$-th (canonical) \emph{numerator})
 and $B_{N}$ (the $N$-th (canonical) \emph{denominator})
 satisfy the following
recurrence relations:
{\allowdisplaybreaks
\begin{align}\label{E:recur}
A_{N}&=b_{N}A_{N-1}+a_{N}A_{N-2},A_{-1}=1, A_0=0,\\
B_{N}&=b_{N}B_{N-1}+a_{N}B_{N-2},B_{-1}=0, B_0=1 \notag .
\end{align}
}
It can also be easily shown that
{\allowdisplaybreaks
\begin{equation}\label{detform}
A_{N}B_{N-1}-A_{N-1}B_{N}=(-1)^{N-1}\prod_{i=1}^{N}a_{i}.
\end{equation}
}

\medskip

We call $d_{0}+K_{n=1}^{\infty}c_{n}/d_{n}$  a \emph{canonical
contraction} of
 $b_{0}+K_{n=1}^{\infty}a_{n}/b_{n}$ if
{\allowdisplaybreaks
\begin{align*}
&C_{k}=A_{n_{k}},& &D_{k}=B_{n_{k}}& &\text{ for }
k=0,1,2,3,\ldots \, ,\phantom{asdasd}&
\end{align*}
}
where $C_{n}$, $D_{n}$, $A_{n}$ and $B_{n}$ are canonical
numerators and denominators of $d_{0}+K_{n=1}^{\infty}c_{n}/d_{n}$
and $b_{0}+K_{n=1}^{\infty}a_{n}/b_{n}$ respectively.

From \cite{LW92} (page 83) we have the following theorem:
\begin{theorem}\label{T:t1}
The canonical contraction of $b_{0}+K_{n=1}^{\infty}a_{n}/b_{n}$
with
{\allowdisplaybreaks
\begin{align*}
&C_{k}=A_{2k}& &D_{k}=B_{2k}& &\text{ for } k=0,1,2,3,\ldots \, ,&
\end{align*}
}
exists if and only if $b_{2k} \not = 0 for k=0,1,2,3,\ldots$, and in
this case is given by
{\allowdisplaybreaks
\begin{equation}\label{E:evcf}
b_{0} + \frac{b_{2}a_{1}}{b_{2}b_{1}+a_{2}} \-
\frac{a_{2}a_{3}b_{4}/b_{2}}{a_{4}+b_{3}b_{4}+a_{3}b_{4}/b_{2}} \-
\frac{a_{4}a_{5}b_{6}/b_{4}}{a_{6}+b_{5}b_{6}+a_{5}b_{6}/b_{4}} \-
\cds .
\end{equation}
}
\end{theorem}
The continued fraction \eqref{E:evcf} is called the \emph{even}
part of $b_{0}+K_{n=1}^{\infty}a_{n}/b_{n}$.

From \cite{LW92} (page 85) we also have:
\begin{theorem}\label{odcf}
The canonical contraction of $b_{0}+K_{n=1}^{\infty}a_{n}/b_{n}$
with $C_{0}=A_{1}/B_{1}$, $D_{0}=1$ and
{\allowdisplaybreaks
\begin{align*}
&C_{k}=A_{2k+1}& &D_{k}=B_{2k+1}& &\text{ for } k=1,2,3,\ldots \,
,&
\end{align*}
}
exists if and only if $b_{2k+1} \not = 0$ for $k=0,1,2,3,\ldots$,
and in this case is given by
{\allowdisplaybreaks
\begin{multline}\label{E:odcf}
\frac{b_{0}b_{1}+a_{1}}{b_{1}} -
\frac{a_{1}a_{2}b_{3}/b_{1}}{b_{1}(a_{3}+b_{2}b_{3})+a_{2}b_{3}}
\-
\frac{a_{3}a_{4}b_{5}b_{1}/b_{3}}{a_{5}+b_{4}b_{5}+a_{4}b_{5}/b_{3}}\\
\- \frac{a_{5}a_{6}b_{7}/b_{5}}{a_{7}+b_{6}b_{7}+a_{6}b_{7}/b_{5}}
\- \frac{a_{7}a_{8}b_{9}/b_{7}}{a_{9}+b_{8}b_{9}+a_{8}b_{9}/b_{7}}
\- \cds .
\end{multline}
}
\end{theorem}
The continued fraction \eqref{E:odcf} is called the \emph{odd}
part of $b_{0}+K_{n=1}^{\infty}a_{n}/b_{n}$.

\medskip
We also make repeated use of
\medskip

\textbf{Worpitzky's Theorem}\emph{ $($\cite{LW92}, pp. 35--36$)$
For all $n \geq 1$, let
\[
|a_{n}|\leq \frac{1}{4}.
\]
 Then $K_{n=1}^{\infty}a_{n}/1$
converges.  All approximants $f_{n}$ of the continued fraction lie in the
disc $|w|<1/2$ and the value of the continued fraction $f$ is in the
disk $|w|\leq1/2$.}

\section{ Background}

\medskip

Of course the idea of using the convergence of the odd and even parts of a
continued fraction $K_{n = 1}^{\infty}\displaystyle{a_{n}/1}$  to show that the
continued fraction itself converges is not new. The following system of inequalities,
called the \emph{fundamental inequalities} by Wall \cite{W48}, ensures the
convergence of the odd and even parts of $K_{n = 1}^{\infty}\displaystyle{a_{n}/1}$.
{\allowdisplaybreaks
\begin{align}\label{fueqs}
r_{1}|1+a_{1}|&\geq |a_{1}|,  \\
 r_{2}|1+a_{1}+a_{2}|&\geq |a_{2}|,\notag  \\
r_{n}|1+a_{n-1}+a_{n}|&\geq r_{n}r_{n-2} |a_{n-1}| +|a_{n}|, \,\,\,\,\,\,\,
n=3,4,5,\dots , \notag
\end{align}
}
where $r_{n} \geq 0$. These inequalities are obtained by applying
the \'Sleszy\'nski-Pringsheim criterion stated below to continued
fractions equivalent to the even and odd parts of $K_{n =
1}^{\infty}\displaystyle{a_{n}/1}$. This approach has been the
starting point for several important lines of research (see
\cite{W48}, Chapter III and \cite{JT80}, section 4.4.5, for more
details of these).

\medskip

 \textbf{\'Sleszy\'nski-Pringsheim Theorem} \emph{$($
\cite{P99}, see also \cite{JT80}, page 92$)$ The continued
fraction $K(a_{n}/b_{n})$ converges to a finite value if
{\allowdisplaybreaks
\begin{align}
&|b_{n}|\geq |a_{n}|+1,& &n=1,2,3, \ldots .
\end{align}
}
If $f_{n}$ denotes the $n$-th approximant, then
{\allowdisplaybreaks
\begin{align}
&|f_{n}|<1,& &n=1,2,3, \ldots .
\end{align}
}
 }

Rather than working with very general implicit inequalities, such as those
at \eqref{fueqs}, we have looked for simple explicit inequalities. This approach
allows us to state some quite general, simple criteria for the convergence
of certain classes of continued fractions of the form
$K_{n = 1}^{\infty}\displaystyle{a_{n}/1}$. We have the following theorem.

{\allowdisplaybreaks }

\section{Main Theorem and Examples}
{\allowdisplaybreaks
\begin{theorem}\label{tgen}
Let $\{c_{n}\}_{n = 1}^{\infty}$ and $\{\beta_{n}\}_{n =
1}^{\infty}$  be sequences of  numbers, with $c_{n}>0$ and $0 <
\beta_{n}<1$ for $n \geq 1$. If the terms in the sequence
$\{a_{n}\}_{n = 1}^{\infty}$ satisfy, for $n \geq 1$, either
{\allowdisplaybreaks
\begin{align}\label{cond1}
|a_{2n}| &\geq \frac{c_{n}+1}{1-\beta_{n}},\\
|a_{2n+1}|&\leq \min \left \{ c_{n}, \,c_{n+1}, \,
 \beta_{n} \beta_{n+1} \frac{c_{n+1}+1}{4(1-\beta_{n+1})}, \,
\beta_{n} \beta_{n+1} \frac{c_{n}+1}{4(1-\beta_{n})} \right \},
\notag
\end{align}
}
or
{\allowdisplaybreaks
\begin{align}\label{cond2}
|a_{2n}| &\geq \max \left \{\frac{c_{n}+1}{1-\beta_{n}},  \frac{c_{n+1}+1}{1-\beta_{n+1}} \right \},\\
|a_{2n+1}|&\leq \min \left \{ c_{n}, \,c_{n+1}, \,
 \beta_{n} \beta_{n+1} \frac{c_{n+1}+1}{4(1-\beta_{n+1})} \right \}, \notag
\end{align}
}
then the continued fraction $K_{n =
1}^{\infty}\displaystyle{a_{n}/1}$ converges to a finite value.
\end{theorem}
}

\begin{proof}
By Theorem \ref{T:t1} the even part of  $K_{n = 1}^{\infty}\displaystyle{a_{n}/1}$
is
 {\allowdisplaybreaks
\begin{align}\label{evprt}
&\phantom{a}\frac{a_{1}}{1+a_{2}}
\-
\frac{a_{2}a_{3}}{a_{4}+1+a_{3}}
\-
\frac{a_{4}a_{5}}{a_{6}+1+a_{5}}
\-
\cds
\-
 \frac{a_{2n}a_{2n+1}}{a_{2n+2}+1+a_{2n+1}}
\-
\cds \\
&\phantom{a} \notag\\
&= \frac{a_{1}/(1+a_{2})}{1}
\-
\frac{a_{2}a_{3}/((1+a_{2})(a_{4}+1+a_{3}))}{1} \notag \\
&\phantom{asdasdas}\-
\frac{a_{4}a_{5}/((a_{4}+1+a_{3})(a_{6}+1+a_{5}))}{1}
\-
\cds \notag \\
&\phantom{asdassfdfdsdas}\-
 \frac{a_{2n}a_{2n+1}/((a_{2n}+1+a_{2n-1})(a_{2n+2}+1+a_{2n+1}))}{1}
\-
\cds . \notag
\end{align}
}
The second continued fraction arises from the first after applying a sequence of
similarity transformations.
We now show that a tail of this continued fraction satisfies the conditions of Worpitzky's Theorem.
From the conditions at \eqref{cond1} or \eqref{cond2} above, it follows that
$|a_{2i}|>c_{i}+1\geq |a_{2i-1}|+1$ and so, using the conditions at
\eqref{cond1} or \eqref{cond2}
 several times, we have that
{\allowdisplaybreaks
\begin{align*}
&\left | \frac{a_{2n}a_{2n+1}}{(a_{2n}+1+a_{2n-1})(a_{2n+2}+1+a_{2n+1})} \right |\\
& \leq
\frac{|a_{2n}| |a_{2n+1}|}{(|a_{2n}|-(c_{n}+1))(|a_{2n+2}|-(c_{n+1}+1))}\\
& \leq
\frac{|a_{2n}| |a_{2n+1}|}{(|a_{2n}|-(1-\beta_{n})|a_{2n}|)(|a_{2n+2}|-(1-\beta_{n+1})|a_{2n+2}|)}\\
&=\frac{|a_{2n+1}| }{\beta_{n}\beta_{n+1}|a_{2n+2}|}
\leq \frac{|a_{2n+1}| }{\beta _{n}\beta_{n+1}(c_{n+1}+1)/(1-\beta_{n+1}) }\leq \frac{1}{4}.
\end{align*}
}
Thus,  by Worpitzky's Theorem,
the even part of $K_{n = 1}^{\infty}\displaystyle{a_{n}/1}$  equals
{\allowdisplaybreaks
\[
\frac{a_{1}/(1+a_{2})}{1}
\-
\frac{a_{2}a_{3}/((1+a_{2})(a_{4}+1+a_{3}))}{1+ \alpha},
\]
}
for some $\alpha$ with  $|\alpha|\leq1/2$.
Likewise,
{\allowdisplaybreaks
\begin{align*}
\left | \frac{a_{2}a_{3}}{(1+a_{2})(a_{4}+1+a_{3})}\right |
&\leq \frac{|a_{3}|}
{(1-1/|a_{2}|)\beta_{2} |a_{4}|} \\
&\leq \frac{\beta_{1}\beta_{2} (c_{2}+1)/(4(1-\beta_{2}))}
{\left (1-\displaystyle{\frac{1-\beta_{1}}
{c_{1}+1}} \right ) \beta_{2}\displaystyle{ \frac{c_{2}+1}{1-\beta_{2}}}} \\
&= \frac{\beta_{1}(c_{1}+1)}{4(c_{1}+\beta_{1})}
<\frac{1}{2}.
\end{align*}
}
Hence the even part of $K_{n = 1}^{\infty}\displaystyle{a_{n}/1}$ converges
to a finite value.

From Theorem \ref{odcf}, the odd part of $K_{n = 1}^{\infty}\displaystyle{a_{n}/1}$ is
{\allowdisplaybreaks
\begin{align}\label{odprt}
&\frac{a_{1}}{1}
-
\frac{a_{1}a_{2}}{a_{3}+1+a_{2}}
\-
\frac{a_{3}a_{4}}{a_{5}+1+a_{4}}
\-
\frac{a_{5}a_{6}}{a_{7}+1+a_{6}}
\-
\cds
\frac{a_{2n-1}a_{2n}}{a_{2n+1}+1+a_{2n}}
\-
\cds \\
&\phantom{a} \notag\\
&\thicksim \frac{a_{1}}{1}
-
\frac{a_{1}a_{2}/(a_{3}+1+a_{2})}{1}
\phantom{}\-
\frac{a_{3}a_{4}/((a_{3}+1+a_{2})(a_{5}+1+a_{4}))}{1}
\-
\cds \notag \\
&\phantom{asdassfdfdsdas}\-
 \frac{a_{2n-1}a_{2n}/((a_{2n-1}+1+a_{2n-2})(a_{2n+1}+1+a_{2n}))}{1}
\-
\cds . \notag
\end{align}
}
By similar reasoning to that used above,
{\allowdisplaybreaks
\begin{align*}
\left |
 \frac{a_{2n-1}a_{2n}}{(a_{2n-1}+1+a_{2n-2})(a_{2n+1}+1+a_{2n})}
\right |
\leq
\frac{|a_{2n-1}|}
{\beta_{n}\beta_{n-1} |a_{2n-2}|}
\leq
\frac{1}{4}.
\end{align*}
}
Thus, again by Worpitzky's theorem,
 the odd part of $K_{n = 1}^{\infty}\displaystyle{a_{n}/1}$ equals
{\allowdisplaybreaks
\[
 \frac{a_{1}}{1}
-
\frac{a_{1}a_{2}/(a_{3}+1+a_{2})}{1+\alpha'},
\]
}
for some $\alpha'$ with $|\alpha'|\leq 1/2$. Thus the odd part of
$K_{n = 1}^{\infty}\displaystyle{a_{n}/1}$ converges to a finite value provided
$|a_{3}+1+a_{2}|>0$, and this follows easily from the
inequalities satisfied by
$|a_{2}|$ and $|a_{3}|$ in the statement of the theorem.

We next show that the conditions at \eqref{cond1} and \eqref{cond2} are also sufficient
to show that the odd and even parts tend to the same limits. From the recurrence
relations at \eqref{E:recur} and Equation \ref{detform}, it follows that
{\allowdisplaybreaks
\begin{align}
\left | \frac{A_{2n+1}}{B_{2n+1}}-\frac{A_{2n-1}}{B_{2n-1}} \right |=
\frac{\prod_{i=1}^{2n}|a_{i}|}{\left | B_{2n+1}B_{2n-1}\right | }.
\end{align}
}
Since the sequence $\{A_{2i+1}/B_{2i+1} \}$ converges to a finite value, it follows that the
expression on the right tends to 0 as $n \to \infty$. Next,
{\allowdisplaybreaks
\begin{align}\label{difeq}
\left | \frac{A_{2n}}{B_{2n}}-\frac{A_{2n-1}}{B_{2n-1}} \right |=
\frac{\prod_{i=1}^{2n}|a_{i}|}{\left | B_{2n}B_{2n-1}\right | }=
\frac{\prod_{i=1}^{2n}|a_{i}|}{\left | B_{2n+1}B_{2n-1}\right | }
\frac{\left | B_{2n+1}\right | }{ \left |B_{2n}\right | }.
\end{align}
}
Thus, if it can be shown that the sequence $\{B_{2n+1}/B_{2n}\}$ is  bounded,
then the left side of \eqref{difeq} tends to 0,
the odd and even parts of $K_{n = 1}^{\infty}\displaystyle{a_{n}/1}$ tend to
the same limits and the continued fraction converges to a finite value.

From the second equation at \eqref{E:recur} and Theorem \ref{T:t1} we have that
{\allowdisplaybreaks
\begin{align*}
&\frac{B_{2n+1}}{B_{2n}}=
1+ \frac{a_{2n+1}}{1}
\+
\frac{a_{2n}}{1}
\+
\frac{a_{2n-1}}{1}
\+
\cds
\+
 \frac{a_{2}}{1}\\
&\phantom{a}\\
&=1+\frac{a_{2n+1}}{1+a_{2n}}
\-
\frac{a_{2n}a_{2n-1}}{a_{2n-2}+1+a_{2n-1}}
\-
\frac{a_{2n-2}a_{2n-3}}{a_{2n-4}+1+a_{2n-3}}
\-\\
&\phantom{sadasfdgdgd}\cds
\-
 \frac{a_{2n-2j}a_{2n-2j-1}}{a_{2n-2j-2}+1+a_{2n-2j-1}}
\-
\cds
\-
\frac{a_{4}a_{3}}{a_{2}+1+a_{3}}
\\
&\phantom{a}\\
&=1+ \frac{a_{2n+1}/(1+a_{2n})}{1}
\-
\frac{a_{2n}a_{2n-1}/((1+a_{2n})(a_{2n-2}+1+a_{2n-1}))}{1}\\
&\phantom{}\-
\frac{a_{2n-2}a_{2n-3}/((a_{2n-2}+1+a_{2n-1})(a_{2n-4}+1+a_{2n-3}))}{1}
\-
\cds\\
&\phantom{}
\-
 \frac{a_{2n-2j}a_{2n-2j-1}/((a_{2n-2j}+1+a_{2n-2j+1})(a_{2n-2j-2}+1+a_{2n-2j-1}))}{1}\\
&\phantom{adssdfsdsdfddgsdfgsdfg}
\-
\cds
\-
\frac{a_{4}a_{3}/((a_{4}+1+a_{5})(a_{2}+1+a_{3}))}{1}
 .
\end{align*}
}
These equalities are valid since the given continued fraction expansion of $B_{2n+1}/B_{2n}$
is finite. Once again using the conditions at  \eqref{cond1} or  \eqref{cond2},
we have that
{\allowdisplaybreaks
\begin{align*}
&\left |
 \frac{a_{2n-2j}a_{2n-2j-1}}{(a_{2n-2j}+1+a_{2n-2j+1})(a_{2n-2j-2}+1+a_{2n-2j-1})}
\right |\\
& \leq
\frac{|a_{2n-2j}| |a_{2n-2j-1}|}{(|a_{2n-2j}|-(c_{n-j}+1))(|a_{2n-2j-2}|-(c_{n-j-1}+1))}\\
& \leq
\frac{|a_{2n-2j}| |a_{2n-2j-1}|}{(|a_{2n-2j}|-(1-\beta_{n-j})|a_{2n-2j}|)
(|a_{2n-2j-2}|-(1-\beta_{n-j-1})|a_{2n-2j-2}|)}\\
&=\frac{|a_{2n-2j-1}| }{\beta_{n-j}\beta_{n-j-1}|a_{2n-2j-2}|}
\leq \frac{|a_{2n-2j-1}| }{\beta _{n-j}\beta_{n-j-1}(c_{n-j-1}+1)/(1-\beta_{n-j-1}) }\leq \frac{1}{4}.
\end{align*}
}
Similarly,
{\allowdisplaybreaks
\begin{align*}
\left |
\frac{a_{2n}a_{2n-1}}{(1+a_{2n})(a_{2n-2}+1+a_{2n-1})}
\right |
&\leq
\frac{|a_{2n-1}|}
{(1-1/|a_{2n}|)\beta_{n-1}|a_{2n-2}|}\\
&\leq
\frac{\beta_{n}\beta_{n-1}(c_{n-1}+1)/(4(1-\beta_{n-1}))}
{\displaystyle{\frac{c_{n}+\beta_{n}}{c_{n}+1}}\beta_{n-1}
\displaystyle{\frac{c_{n-1}+1}{1-\beta_{n-1}}}}\\
&=\frac{\beta_{n}(c_{n}+1)}{4(c_{n}+\beta_{n})}
\leq
\frac{1}{4}
\end{align*}
Thus, by Worpitzky's theorem, there exists $\alpha$ with $|\alpha| \leq 1/2$ such that
\begin{align*}
\left |
\frac{B_{2n+1}}{B_{2n}}
\right |
=
\left |
1+ \frac{a_{2n+1}/(1+a_{2n})}{1+ \alpha}
\right |
\leq 1+2 \frac{\left |a_{2n+1}\right |}{\left |a_{2n}\right | -1}
\leq 3.
\end{align*}
}
Thus the sequence $\{B_{2n+1}/B_{2n}\}$ is  bounded by 3 and
$K_{n = 1}^{\infty}\displaystyle{a_{n}/1}$ converges to a finite value.
\end{proof}
If $\{c_{n}\}$ and $\{\beta_{n}\}$ are increasing sequences, the statement of the theorem is
simplified a little and we get the following corollary.

\begin{corollary}\label{cornew}
Let $\{c_{n}\}_{n = 1}^{\infty}$ and $\{\beta_{n}\}_{n = 1}^{\infty}$  be
increasing  sequences of  numbers,
with $c_{n}>0$ and $0 < \beta_{n}<1$ for $n \geq 1$. If the terms in the sequence
$\{a_{n}\}_{n = 1}^{\infty}$ satisfy, for $n \geq 1$, either
{\allowdisplaybreaks
\begin{align*}&|a_{2n}| \geq \frac{c_{n}+1}{1-\beta_{n}},&
&|a_{2n+1}|\leq \min \left \{ c_{n},
\,
\beta_{n} \beta_{n+1} \frac{c_{n}+1}{4(1-\beta_{n})} \right \}, &\notag
\end{align*}
}
or
{\allowdisplaybreaks
\begin{align*}
&|a_{2n}| \geq  \frac{c_{n+1}+1}{1-\beta_{n+1}},&
&|a_{2n+1}|\leq \min \left \{ c_{n}, \, \beta_{n} \beta_{n+1} \frac{c_{n+1}+1}{4(1-\beta_{n+1})} \right \},
& \notag
\end{align*}
}
then the continued fraction $K_{n = 1}^{\infty}\displaystyle{a_{n}/1}$ converges
to a finite value.
\end{corollary}

In Corollary \ref{cornew} the numbers $(c_{n}+1)/(1-\beta_{n})$,
$c_{n}$ and $\beta_{n} \beta_{n+1} (c_{n}+1)/(4(1-\beta_{n}))$ increase with $n$
and so the corollary gives a convergence criterion for continued fractions
 $K_{n = 1}^{\infty}\displaystyle{a_{n}/1}$ in which both the even- and odd-indexed
partial numerators can become arbitrarily large. The necessity to
find the minimum of $c_{n}$ and $\beta_{n} \beta_{n+1}
(c_{n}+1)/(4(1-\beta_{n}))$ or of $c_{n}$ and $\beta_{n}
\beta_{n+1} (c_{n+1}+1)/(4(1-\beta_{n+1}))$ is a little
cumbersome. We also have the following corollaries which give
cleaner conditions on the $a_{n}$.

It is also of interest to be
able to prove the convergence of continued fractions where both
the odd and even-indexed partial numerators become unbounded.
Many convergence theorems require that infinitely many of the $a_{n}$
lie inside some fixed bounded disc for the continued fraction $K_{n=1}^{\infty}a_{n}/1$
to converge. Hayden's theorem \cite{H62} (see also \cite{JT80}, page 126)
requires at least one
of $a_{n}$, $a_{n+1}$ to lie inside the unit disc, for each $n \geq 1$.
For $K_{n=1}^{\infty}c_{n}^{2}/1$ to converge, Lange's theorem
\cite{L66} (see also \cite{JT80}, page 124) requires $|c_{2n-1}\pm i\,a|\leq \rho$, where
$a$ is a complex number and $\rho$ is real number satisfying
\[
|a|<\rho <|a+1|.
\]
Our theorem allows us to prove the convergence of certain continued
fractions $K_{n=1}^{\infty}a_{n}/1$,
where the sequence $\{a_{n}\}$ does not contain any bounded subsequence.
\begin{corollary}\label{cor2}
Let $\{d_{n}\}_{n = 1}^{\infty}$ be an increasing sequence of positive numbers, with
$d_{n}>25$ for $n \geq 1$. Suppose the terms of the sequence $\{a_{n}\}$ satisfy
{\allowdisplaybreaks
\begin{align*}
|a_{2n}|&\geq d_{n},\\
|a_{2n+1}| & \leq \frac{4}{25}\, d_{n}
\end{align*}
}
Then the continued fraction $K_{n = 1}^{\infty}\displaystyle{a_{n}/1}$ converges
to a finite value.
\end{corollary}
\begin{proof}
For $n \geq 1$, let $c_{n}=d_{n}/5-1$ and $\beta_{n}=4/5$ in Theorem \ref{tgen} and
use the conditions at \eqref{cond1}.
\end{proof}
\begin{example}
Let the terms of the sequence $\{a_{n}\}$ satisfy
\begin{align*}
&|a_{2n-1}|=4n,& &|a_{2n}|= 25n.&
\end{align*}
Then $K_{n=1}^{\infty}a_{n}/1$ converges.
\end{example}

\begin{corollary}\label{cor1}
Let $c>0$. Suppose the terms in the sequence $\{a_{n}\}$ satisfy
{\allowdisplaybreaks
\begin{align*}
|a_{2n}| &\geq 1+3 c+2\sqrt{c}\sqrt{2c+1},\\
|a_{2n+1} | &\leq c.
\end{align*}
}
Then the continued fraction
$K_{n = 1}^{\infty}\displaystyle{a_{n}/1}$ converges to a finite value.
\end{corollary}

\begin{proof} In Theorem \ref{tgen}, let $c_{n}=c$ and
$\beta_{n}=2(\sqrt{c}\sqrt{2c+1}-c)/(c+1)$, for $n=1,2,3, \ldots$ .
\end{proof}
For large $c$ this is clearly a weaker result  than that of Thron \cite{T59}, which states
the following (see \cite{JT80}, page 124):

 \emph{ For $\rho >1$,
$K_{n = 1}^{\infty}\displaystyle{a_{n}/1}$ converges to a finite value provided
that
{\allowdisplaybreaks
\begin{align*}
&|a_{2n-1}| \leq \rho^{2},&   &|a_{2n}| \geq 2(\rho^2-\cos \arg a_{2n}),& &n=1,2,3,\ldots .
\end{align*}
}
}

However, for small $c$ it is possible to prove the convergence of certain
continued fractions whose convergence cannot be proved by Thron's
result. We have the following example.
\begin{example}\label{ex1}
Let the terms of the sequence $\{a_{n}\}$ satisfy
{\allowdisplaybreaks
\begin{align*}
&|a_{2n}| \geq \frac{36}{23},&
&|a_{2n+1} | \leq \frac{1}{23},&
\end{align*}
}
with $a_{2k}=-36/23$ for infinitely many $k$. Then the  continued fraction
$K_{n = 1}^{\infty}\displaystyle{a_{n}/1}$ converges to a finite value.
\end{example}
This follows from Corollary \ref{cor1} with $c=1/23$. Thron's theorem does not
give the convergence of the continued fraction in this example, since there
is no real $\rho >1$ satisfying $36/23 \geq 2(\rho^2+1)$ (Hayden's theorem \cite{H62}
also gives the convergence of this continued fraction).

\allowdisplaybreaks{
}

\end{document}